\documentclass{amsart}
\usepackage{amsmath,amsfonts,mathrsfs,amssymb,enumerate,graphicx,subfig,hyperref,caption}
\usepackage{color} %I added, we can dismiss later if not needed.
\newtheorem{theorem}{Theorem}

\newtheorem{proposition}{Proposition}[section]

\newtheorem{lemma}[proposition]{Lemma}

\DeclareMathOperator{\GL}{GL}

\newcommand{\threebar}[1]{{\left\vert\kern-0.25ex\left\vert\kern-0.25ex\left\vert #1 
    \right\vert\kern-0.25ex\right\vert\kern-0.25ex\right\vert}}

\newcommand{\transp}{\top}

\newcommand{\R}{\mathbb{R}}
\newcommand{\N}{\mathbb{N}}

\newcommand{\Z}{\mathbb{Z}}

\title{On dense intermingling of exact overlaps and the open set condition}
\author{Ian D. Morris}
\address{School of Mathematical Sciences, Queen Mary, University of London, Mile End Road, London E1 4NS, U.K.}
\email{i.morris@qmul.ac.uk }
\begin{document}

\begin{abstract}
We prove that certain families of homogenous affine iterated function systems in $\R^d$ have the property that the open set condition and the existence of exact overlaps both occur densely in the space of translation parameters. These examples demonstrate that in the theorems of Falconer and Jordan-Pollicott-Simon on the almost sure dimensions of self-affine sets and measures, the set of exceptional translation parameters can be a dense set. The proof combines results from the literature on self-affine tilings of $\R^d$ with an adaptation of a classic argument of Erd\H{o}s on the singularity of certain Bernoulli convolutions. Our result encompasses a one-dimensional example due to Kenyon which arises as a special case.
%\\\\MSC Primary: 28A80. Keywords: exact overlap, iterated function system, open set condition, self-affine measure, self-affine set, self-affine tiling.
\end{abstract}
\maketitle
%\setcounter{section}{-1}
%\section{To do}

%Fill in the co-ordinate change in blue in the section section; write the introduction.
%{\blue{Additional material is in blue}}

\section{Introduction and statement of results}

We recall that an \emph{iterated function system} acting on $\R^d$ is a tuple $(T_1,\ldots,T_N)$ of contractions of $\R^d$ with respect to some norm on $\R^d$, which in this article will always be the Euclidean norm. It is classical (see \cite{Hu81}) that for every such iterated function system there exists a unique nonempty compact set $X\subset \R^d$ with the property $X=\bigcup_{j=1}^N T_jX$ and this set is conventionally called the \emph{attractor} or \emph{limit set} of the iterated function system. The iterated function system $(T_1,\ldots,T_N)$ is said to satisfy the \emph{open set condition} if there exists a nonempty open set $U\subset \R^d$ such that $T_jU\subseteq U$ for every $j=1,\ldots,N$ and such that the distinct images $T_j U$, $T_kU$ have empty pairwise intersection. If the open set condition holds and every $T_j$ is given by a similarity transformation $T_jx=r_j O_jx+v_j$ with $r_j \in (0,1)$, $O_j \in O(d)$ and $v_j \in \R^d$, then it is well known that the Hausdorff dimension and upper and lower box dimensions of $X$ coincide and are all equal to the unique $s \in (0,d]$ such that $\sum_{j=1}^N r_j^s=1$, a value which is conventionally called the \emph{similarity dimension} of the iterated function system.
Since the classic work of J.E. Hutchinson in \cite{Hu81} there has been substantial interest in extending this result by weakening either the hypothesis of the open set condition, the hypothesis that the transformations $T_j$ should be similitudes, or both.

An iterated function system $(T_1,\ldots,T_N)$ acting on $\R^d$ is said to have \emph{exact overlaps} if for some pair of distinct words $j_1j_2 \cdots j_n \in \{1,\ldots,N\}^n$ and $k_1k_2\cdots k_m \in \{1,\ldots,N\}^m$ one has $T_{j_1}T_{j_2}\cdots T_{j_n}=T_{k_1}T_{k_2}\cdots T_{k_m}$. By considering the words $j_1\cdots j_n k_1\cdots k_m$ and $k_1\cdots k_mj_1\cdots j_n$ in place of $j_1\cdots j_n$ and $k_1\cdots k_m$ it may be assumed without loss of generality that $n=m$, and we will always make this reduction in the sequel. When $(T_1,\ldots,T_N)$ consists of similitudes and admits exact overlaps it is well known as folklore (and is explicitly articulated in \cite{SiSo02}) that the Hausdorff and box dimensions of the attractor are strictly smaller than the similarity dimension. In this work we will call an iterated function system for which the Hausdorff dimension of the attractor is smaller than the similarity dimension \emph{exceptional}. It is conjectured (see e.g. \cite{Ho20}) that for an iterated function system of similitudes the \emph{only} situations in which the system is exceptional are those in which either an exact overlap exists, or the similarity dimension exceeds the dimension of the smallest affine subspace of $\R^d$ which is preserved by every $T_j$. Both of the latter two circumstances are precluded \emph{a priori} when the open set condition is satisfied.

More generally, an \emph{affine iterated function system} is usually understood to be a tuple $(T_1,\ldots,T_N)$ of invertible affine transformations of $\R^d$ all of which are contracting with respect to some norm on $\R^d$. In this context a suitable analogue of the similarity dimension, now often called the \emph{affinity dimension}, was articulated by Falconer in \cite{Fa88}. The affinity dimension is always an upper bound for the upper box dimension of the attractor. According to one of the principal results of \cite{Fa88}, if a tuple of linear maps $A_1,\ldots,A_N \in \GL_d(\R)$ is specified and those maps are all sufficiently strongly contracting, then for  Lebesgue almost every choice of $(u_1,\ldots,u_N) \in (\R^d)^N$ the attractor of the iterated function system $(T_1,\ldots,T_N)$ defined by $T_jx:=A_jx+u_j$ has Hausdorff dimension and upper and lower box dimension equal to its affinity dimension.  On the other hand affine iterated function systems for which these dimensions are \emph{not} equal to the affinity dimension (which we again describe as ``exceptional'') are much easier to construct than in the self-similar case, as follows.

In the affine case the open set condition does not suffice to guarantee that the dimension of the attractor is nonzero (an observation which dates back to \cite{Ed92}). In the case of affine iterated function systems the existence of exact overlaps again guarantees exceptionality, although the required arguments are much more delicate than in the self-similar case and are correspondingly much more recent (see the introduction to \cite{Mo22}). As before, if an affine iterated function system preserves an affine subspace with dimension smaller than the affinity dimension then the affine subspace necessarily contains the attractor, so the dimension of the subspace bounds the dimension of the attractor and the iterated function system is necessarily exceptional. In the affine case a further possibility exists that the iterated function system may lack an invariant affine subspace but preserve a low-dimensional variety or manifold, a phenomenon which is investigated in \cite{BaKr11,FeKa18}. For example, if $T_1,\ldots,T_N \colon \R^3\to \R^3$ are contractions of the form
\[T\begin{pmatrix} x\\y\\z\end{pmatrix}:=\begin{pmatrix} \lambda_1\lambda_2& \lambda_1 v & \lambda_2 u \\ 0&\lambda_1&0\\ 0&0&\lambda_2 \end{pmatrix}\begin{pmatrix} x\\y\\z\end{pmatrix} + \begin{pmatrix} uv\\u\\v\end{pmatrix}\]
then the two-dimensional variety $\{(x,y,z) \in \R^3 \colon x=yz\}$ is preserved by the iterated function system $(T_1,\ldots,T_N)$, hence this variety contains the attractor and the dimension of the attractor cannot be higher than two. On the other hand as long as the parameters $\lambda_1$, $\lambda_2$, $u$, $v$ are chosen such that the fixed points of $T_1,\ldots,T_N$ do not belong to a common one- or two-dimensional affine subspace of $\R^3$ then no invariant affine subspace can exist, since any invariant subspace must contain these fixed points. Lastly, it may happen that  the open set condition holds, no proper submanifold of $\R^d$ is preserved by the iterated function system, no exact overlaps exist, and yet the iterated function system is nonetheless exceptional (see e.g. \cite{Be84,DaSi17,FeWa05,Fr12,Mc84}). A feature common to most such systems is the existence of a linear projection map $P\colon \R^d \to \R^d$ such that $PT_jP=PT_j$ for every $j=1,\ldots,N$ and such that the kernel of the projection $P$ is more strongly contracted by the transformations $T_j$ than is the image subspace of $P$. Remarkably, at least three of these four sources of exceptionality may hold robustly with respect to changes in the translation parameters $v_1,\ldots,v_N$ of the affine transformations $T_ix=A_ix+v_i$; see the introduction to \cite{Mo22} for further examples and discussion.

This note is concerned with the way in which three of the properties described above --  the open set condition, the existence of exact overlaps, and the property of being exceptional -- depend on the translation parameters of an affine iterated function system. More formally, given linear maps $A_1,\ldots,A_N \in \GL_d(\R)$ all of which are contracting with respect to some norm on $\R^d$, we ask about the structure of the set of tuples $(u_1,\ldots,u_N) \in (\R^d)^N$ such that the iterated function system $(T_1,\ldots,T_N)$ defined by $T_jx:=A_j+u_j$ is exceptional (or otherwise), or satisfies the open set condition, or admits exact overlaps. In this note we exhibit examples with the property that all of these phenomena are at once dense in the space of translation parameters. This family of examples may be constructed for every value of $d$, and in two dimensions includes \emph{strongly irreducible} examples for which the linear maps $A_1,\ldots,A_N$ do not preserve any finite union of lines through the origin in $\R^d$.

The purpose of this note is to prove the following theorem:
\begin{theorem}\label{th:main}
Let $A$ be a $d \times d$ matrix with integer entries and with all eigenvalues strictly larger than $1$ in modulus, define $N:=|\det A|$ and suppose that $N>2$. For each $u_1,\ldots,u_N \in \R^d$ define an iterated function system $(T_1,\ldots,T_N)$ acting on $\R^d$ by $T_jx:=A^{-1}x+u_j$ for every $j=1,\ldots,N$, and let $\nu$ denote the unique Borel probability measure on $\R^d$ which satisfies $\nu=\frac{1}{N}\sum_{j=1}^N(T_j)_*\nu$.  If $u_1,\ldots,u_N \in \bigcup_{m \geq 0} A^{-m}\Z^d$ then the associated iterated function system $(T_1,\ldots,T_N)$ satisfies exactly one of the following:
\begin{enumerate}[(i)]
\item\label{it:th1}
The open set condition is satisfied, the attractor has positive Lebesgue measure and is equal to the closure of its interior, and $\nu$ is the normalised restriction of Lebesgue measure to the attractor. 
\item\label{it:th2}
Exact overlaps exist, the attractor has upper box dimension strictly smaller than $2$ and the measure $\nu$ is singular.
\end{enumerate}
Additionally each of \eqref{it:th1} and \eqref{it:th2} is satisfied for a dense set of $(u_1,\ldots,u_N) \in  (\bigcup_{m \geq 0} A^{-m}\Z^d)^N$ and in particular holds for a dense set of $(u_1,\ldots,u_N) \in (\R^d)^N$.
\end{theorem}
If $\|A^{-1}\|<\frac{1}{2}$ then the theorem of Falconer from \cite{Fa88} (in its modified form which was introduced by Solomyak in  \cite{So98}) may be applied to the above system to show that for Lebesgue almost every $(u_1,\ldots,u_N) \in (\R^d)^N$ the attractor of $(T_1,\ldots,T_N)$ has Hausdorff dimension precisely $d$. Likewise a theorem of Jordan, Pollicott and Simon in \cite{JoPoSi07} applies to show that the dimension of the uniform self-affine measure $\nu$ defined by $(T_1,\ldots,T_N)$ is equal to $d$ for almost all translation parameters.  A recent result of D.-J. Feng \cite{Fe19} also implies that the set of non-exceptional parameters $(u_1,\ldots,u_N)$ in both of these theorems contains a dense $G_\delta$ set as well as having full Lebesgue measure. Theorem \ref{th:main} thus illustrates on the other hand that the set of exceptional translation parameters in these two theorems is in some cases dense in the parameter space, a possibility which does not seem to have been explicitly noted before.  We remark that the special case of Theorem \ref{th:main} in which $d=1$ and in which $A$ is equal to the $1\times 1$ matrix with $3$ as its sole entry was previously obtained by R. Kenyon in \cite[\S2]{Ke97}, but its connection with Falconer's theorem was not remarked on in that work.

An iterated function system $(T_1,\ldots,T_N)$ of the form $T_jx=A_ix+v_i$ is often called \emph{reducible} if the matrices $A_i$ all preserve a common subspace $V$ of $\R^d$ such that $0<\dim V<d$ and \emph{irreducible} otherwise; the system is called \emph{strongly irreducible} if the matrices $A_i$ do not preserve any finite union of such subspaces. Strong irreducibility is a requirement of several known sufficient conditions for the Hausdorff dimension of the attractor to equal the affinity dimension (most notably \cite{BaHoRa19,HoRa19} as well as a number of now-subsumed earlier results) and examples where the Hausdorff dimension does \emph{not} equal the affinity dimension are frequently reducible (see for example \cite{Be84,DaSi17,FeWa05,Mc84,PrUr89}) or at least fail to be strongly irreducible (\cite{Fr12}).  When $d>2$ the matrix $A$ in Theorem \ref{th:main} necessarily admits a real invariant subspace and in those cases the iterated function system $(T_1,\ldots,T_N)$ is consequently reducible. On the other hand when $d=2$ strong irreducibility occurs whenever $A$ has non-real eigenvalues which are not proportional to roots of unity. A strongly irreducible example with $d=2$ and $\|A^{-1}\|=1/\sqrt{N}=1/\sqrt{5}$ is illustrated in Figure \ref{fi:only}.

\begin{figure}%
    \begin{center}
    \subfloat[(i). A strongly irreducible planar system satisfying the open set condition.]{{\includegraphics[width=5.9cm]{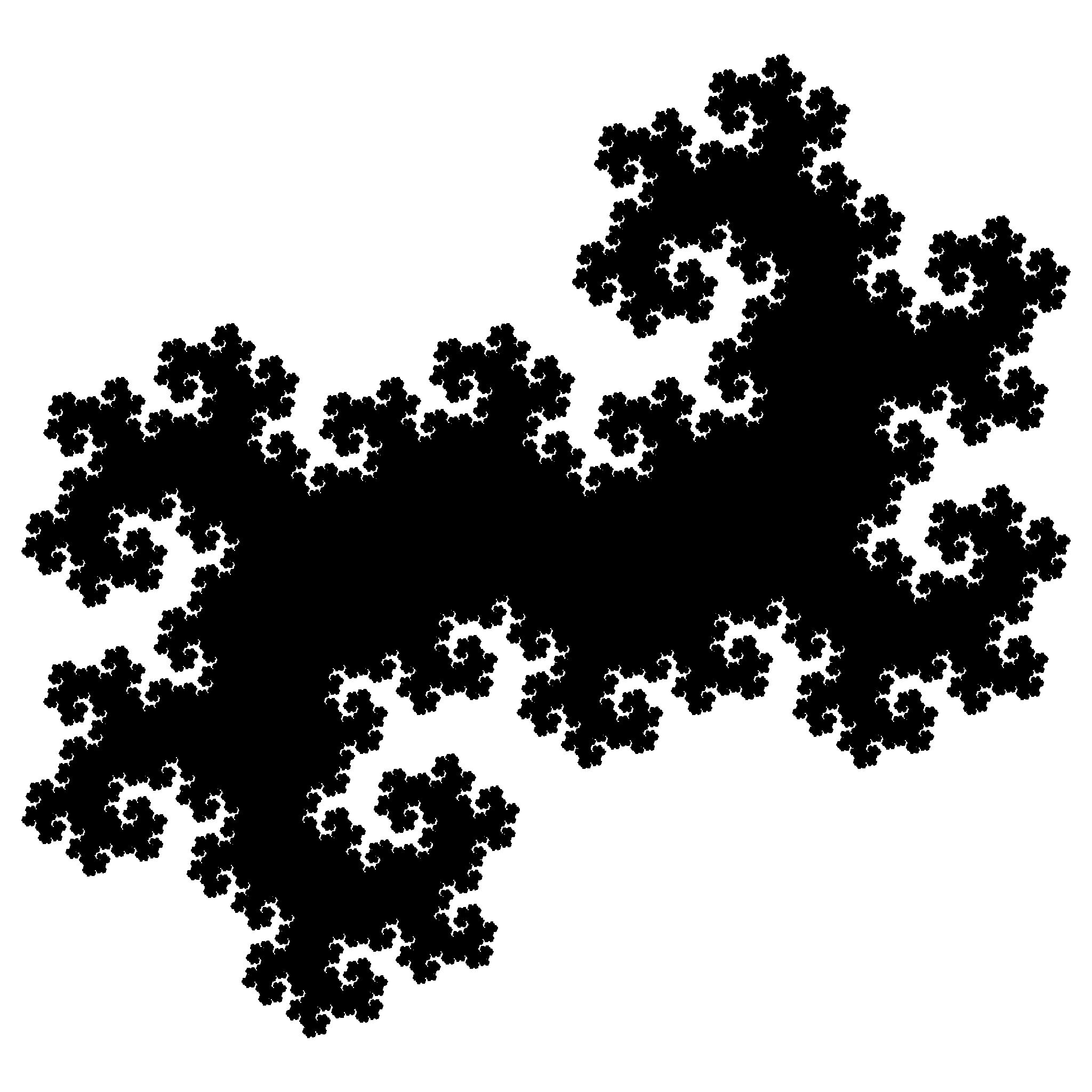}}}
    \quad\quad
    \subfloat[(ii). The same system with different translation parameters and with exact overlaps.]{{\includegraphics[width=5.9cm]{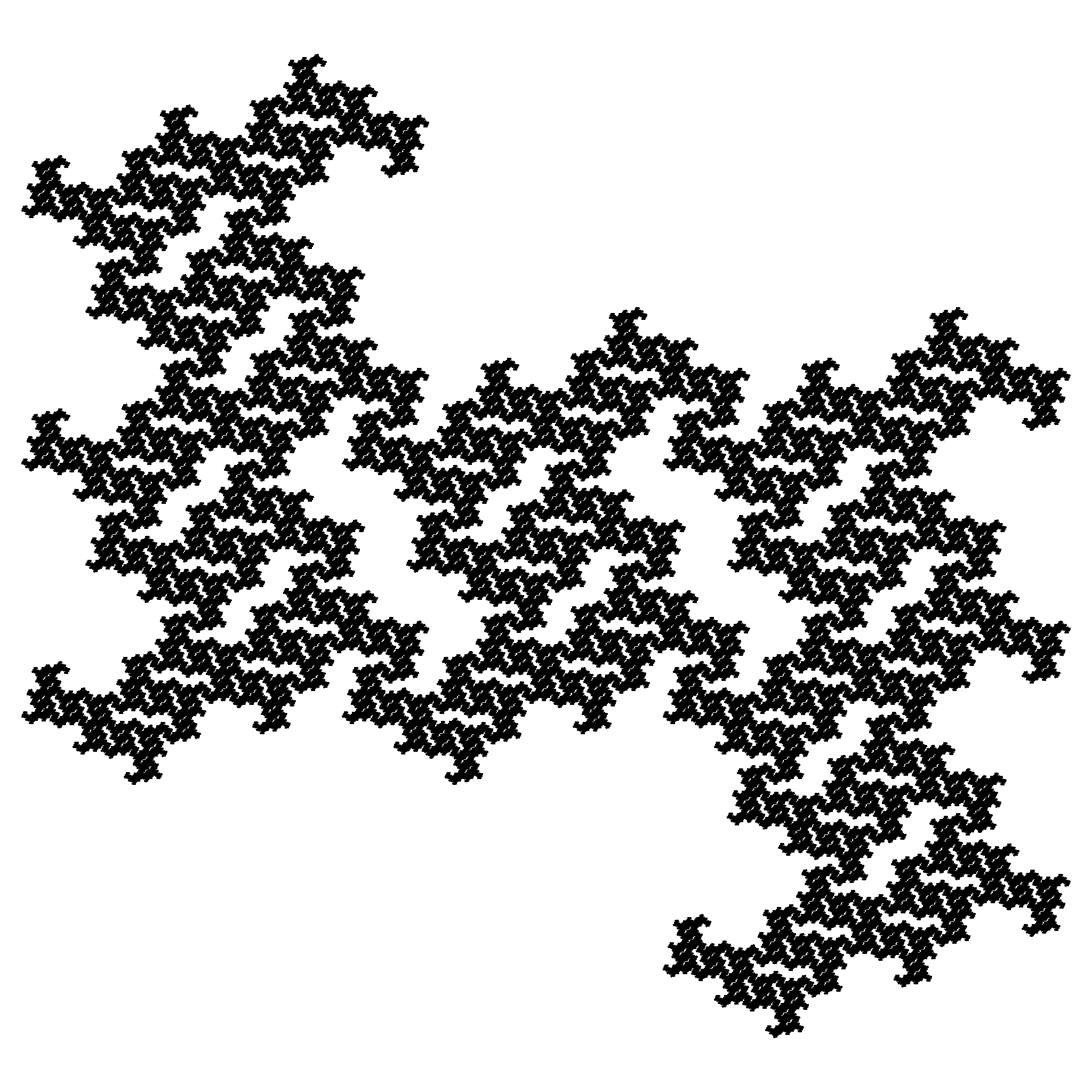}}}
    \caption[.]{Attractors of two iterated function systems of the type described in Theorem \ref{th:main}
    \label{fi:only}. In the first example we have\medskip\leavevmode\\
    \begin{minipage}{\linewidth}
%     \begin{align*}
   % A&=\begin{pmatrix} 1&-2\\2&1\end{pmatrix},\quad &u_1&=\begin{pmatrix}-1\\-1\end{pmatrix}, \quad &u_2&=\begin{pmatrix}-1\\0\end{pmatrix},\\
  %  u_3&=\begin{pmatrix}0\\0\end{pmatrix}  ,\quad &u_4&=\begin{pmatrix}1\\0\end{pmatrix} ,\quad &u_5&=\begin{pmatrix}1\\1\end{pmatrix}.\end{align*}
\[
    A=\begin{pmatrix} 1&-2\\2&1\end{pmatrix},\quad u_1=\begin{pmatrix}-1\\-1\end{pmatrix}, \quad u_2=\begin{pmatrix}-1\\0\end{pmatrix},\]
   \[ u_3=\begin{pmatrix}0\\0\end{pmatrix}  ,\quad u_4=\begin{pmatrix}1\\0\end{pmatrix} ,\quad u_5=\begin{pmatrix}1\\1\end{pmatrix}.\]
\end{minipage}

\medskip
   In the second example the same parameters are used except that the sign of the second entry of each $u_j$ is reversed.
    }
    \end{center}%
\end{figure}

The method of proof of Theorem \ref{th:main} is as follows. Firstly, by an affine change of co-ordinates every iterated function system of the form $T_jx \equiv A^{-1}x+u_j$ with translation vectors in $ \bigcup_{m \geq 0}A^{-m}\Z^d$ is equivalent to such a system with every $u_j$ belonging to $\Z^d$ and with one of the vectors $u_j$ being zero. In this context it is classical that positive Lebesgue measure of the attractor is equivalent to the absence of exact overlaps (see \cite{LaWa96a}) and the remaining details of the dichotomy into \eqref{it:th1} versus \eqref{it:th2} may also be completed using existing results on iterated function systems. Classical results on self-affine tilings such as \cite{Ba91} also provide a sufficient condition for the open set condition which can be shown to hold densely on  $(\bigcup_{m \geq 0}A^{-m}\Z^d)^N$ and this establishes \eqref{it:th1}. To show that \eqref{it:th2} is satisfied densely in $(\bigcup_{m \geq 0}A^{-m}\Z^d)^N$ we modify a classical Fourier-analytic argument of Erd\H{o}s from \cite{Er39} to show that for a dense set of translation parameters the measure $\nu$ is  singular.

We lastly remark that the literature on self-affine tilings includes general sufficient conditions for the attractor to have zero Lebesgue measure (and hence for the case \eqref{it:th2} to hold) which could in principle be used to give an alternative, more constructive proof that \eqref{it:th2} holds densely in certain cases, but these results require additional hypotheses on $A$ and $N$ which would significantly narrow the scope of the result. Such sufficient conditions are part of the subject of the article \cite{LaWa96b}.

\section{Proof of Theorem \ref{th:main}}

Throughout the proof we fix $A$ and $N$ satisfying the hypotheses of Theorem \ref{th:main}. We begin by establishing the dichotomy described in Theorem \ref{th:main}:
\begin{proposition}\label{pr:key}
Let $u_1,\ldots,u_{N} \in \bigcup_{m \geq 0} A^{-m}\Z^d$ and define an iterated function system $(T_1,\ldots,T_N)$ on $\R^d$ by $T_jx:=A^{-1}x+u_j$. Then there exists a unique Borel probability measure on $\R^d$ such that $\frac{1}{N}\sum_{j=1}^N (T_j)_*\nu=\nu$, and exactly one of the following holds:
  \begin{enumerate}[(i)]
  \item\label{it:pr1}
  The iterated function system $(T_1,\ldots,T_N)$ satisfies the open set condition, its attractor has positive Lebesgue measure and is equal to the closure of its interior, and the measure $\nu$ is given by the normalised restriction of Lebesgue measure to the attractor.
  \item\label{it:pr2}
  The  iterated function system $(T_1,\ldots,T_N)$ has exact overlaps, the upper box dimension of its attractor is strictly smaller than $d$, and the measure $\nu$ is singular.
  \end{enumerate} 
  \end{proposition}
  \begin{proof}
  The existence and uniqueness of $\nu$ follow from the classical results of J.E. Hutchinson \cite{Hu81}. Since there exists $m\geq 0$ such that $u_1,\ldots,u_N \in A^{-m}\Z^d$, by a linear change of co-ordinates of the form $x \mapsto A^mx$ we may without loss of generality suppose that $u_1,\ldots,u_N \in \Z^d$; by a further translation of the co-ordinates we also assume without loss of generality that $u_1=0$. It is straightforward to check that the two outcomes \eqref{it:pr1} and \eqref{it:pr2} are unaffected by these co-ordinate transformations. We therefore make this reduction for the remainder of the proof.
  
We first show that if exact overlaps are present then the other properties described in \eqref{it:pr2} above must follow. We recall that the \emph{affinity dimension} defined by Falconer in \cite{Fa88} is a real number associated to each invertible affine iterated function system $(\hat{T}_1,\ldots,\hat{T}_{\hat{N}})$ on $\R^d$ whose properties include the following. Let each $\hat{T}_j$ have the form $\hat{T}_jx=B_jx+v_j$ for some invertible linear map $B_j$ and vector $v_j$. Then the upper box dimension of the attractor of   $(\hat{T}_1,\ldots,\hat{T}_{\hat{N}})$  is bounded above by the affinity dimension of $(\hat{T}_1,\ldots,\hat{T}_{\hat{N}})$, and its affinity dimension is strictly less than $d$ if and only if $\sum_{j=1}^{\hat{N}} |\det B_j|<1$. Suppose now that $(T_1,\ldots,T_N)$ admits an exact overlap corresponding to two words of length $n$, say, and let $X=\bigcup_{j=1}^N T_jX$ be its attractor. Let $\hat{T}_1,\ldots,\hat{T}_{N^n}$ denote the $N^n$ maps of the form $T_{j_1}\cdots T_{j_n}$ where $j_1,\ldots,j_n \in \{1,\ldots,N\}$. Since $(T_1,\ldots,T_N)$ is assumed to have an exact overlap of length $n$, at least two of the maps $\hat{T}_j$ fail to be distinct, so by relabelling if necessary we assume without loss of generality that $\hat{T}_{N^n}=\hat{T}_{N^n-1}$. By iterating the functional equation $X=\bigcup_{j=1}^N T_jX$ we have
\[X=\bigcup_{j=1}^N T_jX= \bigcup_{j_1,\ldots,j_n=1}^N T_{j_1}\cdots T_{j_n}X = \bigcup_{j=1}^{N^n} \hat{T}_jX=\bigcup_{n=1}^{N^n-1}\hat{T}_jX\]
where we have used the fact that the sets $\hat{T}_{N^n}X$ and $\hat{T}_{N^n-1}X$ are identical. Thus $X$ is the attractor of the iterated function system $(\hat{T}_1,\ldots,\hat{T}_{N^n-1})$ as well as of $(T_1,\ldots,T_N)$. But each $\hat{T}_j$ has the form $\hat{T}_jx=A^{-n}x + v_j$ for some $v_j \in \R^d$, and since $\sum_{j=1}^{N^n-1}|\det A^{-n}|=(N^n-1)|\det A|^{-n} = 1-N^{-n}<1$ the affinity dimension of  $(\hat{T}_1,\ldots,\hat{T}_{N^n-1})$ is strictly less than $d$. It follows that the upper box dimension of $X$ is strictly smaller than $d$ and in particular $X$ has zero Lebesgue measure. The support $Z$ of $\nu$ is easily seen to satisfy $Z=\bigcup_{j=1}^N T_jZ$ and is therefore precisely equal to the attractor $X$, so the support of $\nu$ has zero Lebesgue measure and $\nu$ is necessarily singular. This completes the verification of \eqref{it:pr2}.

We now suppose for the remainder of the proof that exact overlaps are not present and demonstrate that in this case \eqref{it:pr1} holds. For this we will apply a result of Lagarias and Wang from the literature on self-affine tilings (\cite{LaWa96a}). A simple induction shows that for every $n \geq 1$, for every word $ j_1j_2\cdots j_n$ and every $x \in \R^d$ we have
  \[T_{j_1}\cdots T_{j_n} x =A^{-n}x + \sum_{r=1}^n A^{1-r} u_{j_r}.\]
 It follows that if $j_1j_2\cdots j_n$ and $k_1k_2\cdots k_n$ are distinct words, then since by hypothesis $T_{j_1}\cdots T_{j_n} \neq T_{k_1}\cdots T_{k_n}$ we necessarily have
 \[ \sum_{r=1}^n A^{1-r} u_{j_r} \neq  \sum_{r=1}^n A^{1-r} u_{k_r}.\]
Multiplying on the left by $A^{n-1}$ and rearranging we find that
 \[\sum_{t=0}^{n-1} A^t u_{j_{n-t}} \neq \sum_{t=0}^{n-1}A^t u_{k_{n-t}}.\]
 Since both expressions belong to $\Z^d$ they are necessarily separated by a Euclidean distance of at least $1$. For every $n \geq 1$ define  \[\mathcal{D}_n:=\left\{\sum_{t=0}^{n-1} A^t u_{j_{t+1}} \colon j_1,j_2,\ldots,j_n \in \{1,\ldots,N\} \right\}.\]
 It follows from the preceding arguments that each $\mathcal{D}_n$ contains exactly $N^n$ distinct points. The fact that $u_1=0$ implies that $\mathcal{D}_n \subseteq \mathcal{D}_{n+1}$ for every $n \geq 1$ since every element of $\mathcal{D}_n$ may be realised as an element of $\mathcal{D}_{n+1}$ with $u_{j_{n+1}}=u_1=0$. The union $\bigcup_{n=1}^\infty \mathcal{D}_n$ consists of points separated by a Euclidean distance of at least $1$ and is thus uniformly discrete in the sense of \cite[\S2]{LaWa96a}. We may now apply the implication (iv)$\implies$(iii) of \cite[Theorem 1.1]{LaWa96a} to deduce that the attractor of $(T_1,\ldots,T_N)$ is equal to the closure of its interior and that the topological boundary of the attractor has zero Lebesgue measure. Let $U$ denote the interior of the attractor: clearly $U$ has positive Lebesgue measure and $U=\bigcup_{j=1}^N T_j U$ up to Lebesgue measure zero. The latter is only possible if both $T_jU \subseteq U$ for every $j=1,\ldots,N$ and if distinct sets $T_jU$ do not intersect, which is precisely the open set condition. It is straightforward to deduce that if $\lambda$ denotes the normalised restriction of Lebesgue measure to $U$ then $\frac{1}{N}\sum_{j=1}^N (T_j)_*\lambda=\lambda$ and this implies that $\lambda=\nu$ since this functional equation has a unique solution. The proof is complete.\end{proof}
 
 To prove Theorem \ref{th:main} it remains to establish that the two outcomes \eqref{it:pr1} and \eqref{it:pr2} each hold for a dense set of parameters $(u_1,\ldots,u_N) \in (\bigcup_{m\geq 0}A^{-m}\Z^d)^N$. In view of the above proposition it will suffice to establish the following two facts: there exists a dense set of $(u_1,\ldots,u_N) \in (\bigcup_{m\geq 0}A^{-m}\Z^d)^N$ such that the associated iterated function system does not have exact overlaps; and there exists a dense set of $(u_1,\ldots,u_N) \in (\bigcup_{m\geq 0}A^{-m}\Z^d)^N$ such that the associated measure $\nu$ is singular. We begin with the former, for which we modify a criterion established by C. Bandt \cite{Ba91}.
\begin{lemma}
Let $\mathcal{U}$ denote the set of all $(u_1,\ldots,u_N) \in (\bigcup_{m\geq 0} A^{-m}\Z^d)^N$ with the property that for some integer $m_0 \in \Z$  we have $A^{m_0}u_j \in \Z^d$ for every $j=1,\ldots,N$ and the $N$ sets $A^{m_0}u_j+A\Z^d$ are all pairwise distinct. If $(u_1,\ldots,u_N) \in \mathcal{U}$, then the associated iterated function system $(T_1,\ldots,T_N)$ defined by $T_jx:=A^{-1}x+u_j$ does not have exact overlaps.
\end{lemma} 
 \begin{proof}
Let $(u_1,\ldots,u_N) \in  (\bigcup_{m\geq 0} A^{-m}\Z^d)^N$ and suppose that $A^{m_0}u_j \in \Z^d$ for every $j=1,\ldots,N$ and that the $N$ sets $A^{m_0}u_j+A\Z^d$ are all pairwise distinct.  Suppose for a contradiction that $T_{j_1}\cdots T_{j_n}=T_{k_1}\cdots T_{k_n}$ for two distinct words $j_1j_2\cdots j_n$ and $k_1k_2\cdots k_n$, which without loss of generality we assume to have equal length.  Since for all $x \in \R^d$
\[T_{j_1}\cdots T_{j_n}x = A^{-n}x + \sum_{r=1}^n A^{1-r}u_{j_r},\qquad T_{k_1}\cdots T_{k_n}x = A^{-n}x + \sum_{r=1}^n A^{1-r}u_{k_r}\]
it follows that
\[ \sum_{r=1}^n A^{1-r}u_{j_r} = \sum_{r=1}^n A^{1-r}u_{k_r}.\]
Let $t \in \{1,\ldots,n\}$ be the largest integer such that $u_{j_t} \neq u_{k_t}$. Clearly this implies
\[ \sum_{r=1}^t A^{1-r}u_{j_r} = \sum_{r=1}^t A^{1-r}u_{k_r}\]
and therefore
\[ \sum_{r=0}^{t-1} A^r u_{j_{t-r}} = \sum_{r=0}^{t-1} A^r u_{k_{t-r}}\]
so that
\[A^{m_0}u_{j_t} - A^{m_0} u_{k_t} = \sum_{r=1}^{t-1} A^{r} \left(A^{m_0}u_{k_r}-A^{m_0}u_{j_r}\right) \in A\Z^d.\]
Hence $A^{m_0}u_{j_t}+A\Z^d =  A^{m_0} u_{k_t}+A\Z^d$, a contradiction. The proof is complete.
 \end{proof}
 
 \begin{lemma}
 The set $\mathcal{U}$ is dense in $(\bigcup_{m\geq 0} A^{-m}\Z^d)^N$.
 \end{lemma}
  \begin{proof}
  Let $(v_1,\ldots,v_N) \in (\bigcup_{m\geq 0} A^{-m}\Z^d)^N$ be arbitrary and choose $m_0 \geq 0$ such that $A^{m_0}v_j \in \Z^d$ for every $j=1,\ldots,N$. Let $n \geq 1$ be arbitrary. Since $\Z^d$ is $(1/\sqrt{2})$-dense in $\R^d$, the lattice $A\Z^d$ is $(\|A\|/\sqrt{2})$-dense in $\R^d$, so the $N$ distinct sets of the form $v+A\Z^d$ are each $(\|A\|/\sqrt{2})$-dense in $\R^d$. It follows that we may choose $(w_1,\ldots,w_N) \in \Z^d$ such that $\max_j \|A^{m_0+n}v_j-w_j\| \leq \|A\|/\sqrt{2}$ and such that the sets $w_j+A\Z^d$ are all distinct. Now define $u_j:=A^{-m_0-n}w_j$ for every $j=1,\ldots,N$. Clearly we have $(u_1,\ldots,u_N) \in \mathcal{U}$ and
  \[\max_j \|u_j-v_j\|\leq \left\|A^{-m_0-n}\right\| \cdot \max_j \left\|A^{m_0+n}v_j-w_j\right\|  \leq \left(\frac{\|A\|}{\sqrt{2}}\right) \left\|A^{-m_0-n}\right\|.\]
  In particular by choosing $n$ sufficiently large the tuple $(u_1,\ldots,u_N)$ may be chosen as close to $(v_1,\ldots,v_N)$ as desired. The lemma is proved.
  \end{proof}
  This completes the proof that the set of parameters corresponding to the outcome \eqref{it:th1} of Theorem \ref{th:main} is dense in $(\bigcup_{m \geq 0} A^{-m}\Z^d)^N$. To show that the second outcome   \eqref{it:th2} is also dense we require two further lemmas. The first of these shows that a certain Fourier-analytic condition holds densely in $(\bigcup_{m \geq 0} A^{-m}\Z^d)^N$, and the second demonstrates that the Fourier-analytic condition implies the singularity of the measure $\nu$.
\begin{lemma}
Let $w \in \Z^d$ be nonzero and let $\mathcal{V}_w$ denote the set of all $(u_1,\ldots,u_N) \in (\bigcup_{m \geq 0}A^{-m}\Z^d)^N$ such that for every $n \in \Z$ we have $\sum_{j=1}^N e^{2\pi i \langle A^n u_j,w\rangle}\neq 0$. Then $\mathcal{V}_w$ is dense in $(\bigcup_{m \geq 0}A^{-m}\Z^d)^N$.
  \end{lemma}
  \begin{proof}
  Fix a nonzero vector $w \in \Z^d$ throughout the proof, and also choose vectors $u_1,\ldots,u_{N-1} \in \bigcup_{m \geq 0}A^{-m}\Z^d$ arbitrarily. We will show that the set of all $u_N \in \bigcup_{m \geq 0}A^{-m}\Z^d$ satisfying $(u_1,\ldots,u_N) \in \mathcal{V}_w$ is a dense set, and this clearly suffices to prove the lemma.
  
  Since $\bigcup_{m \geq 0}A^{-m}\Z^d$ is dense in $\R^d$ it is  sufficient for us to show that the set
  \[\mathcal{G}:=\left\{v \in \R^d \colon e^{2\pi i \langle A^nv,w\rangle} + \sum_{j=1}^{N-1} e^{2\pi i \langle A^nu_j,w\rangle } \neq 0 \text{ for all }n \in \Z\right\}\]
  is open and dense in $\R^d$.   For every $n \in \Z$ define 
  \[\mathcal{G}_n:=\left\{v \in \R^d \colon e^{2\pi i \langle A^nv,w\rangle} + \sum_{j=1}^{N-1} e^{2\pi i \langle A^nu_j,w\rangle}\neq 0\right\}\]
 so that $\mathcal{G}=\bigcap_{n \in \Z} \mathcal{G}_n$. If $|\sum_{j=1}^{N-1} e^{2\pi i \langle A^nu_j,w\rangle}|\neq 1$ then $\mathcal{G}_n$ is equal to $\R^d$; otherwise, the complement of $\mathcal{G}_n$ in $\R^d$ is the union of countably infinitely many evenly-spaced, parallel hyperplanes in $\R^d$, each being a translated copy of the orthogonal complement of $(A^\transp)^nw$. In either case $\mathcal{G}_n$ is easily seen to be open and dense in $\R^d$. Moreover, if the second case occurs only for finitely many $n \in \Z$ then $\mathcal{G}$ is equal to the intersection of finitely many open and dense subsets of $\R^d$ and hence will be open and dense. But it is indeed the case that $|\sum_{j=1}^{N-1} e^{2\pi i \langle A^nu_j,w\rangle}|= 1$ for at most finitely many $n$. To see this we observe that on the one hand, for all sufficiently large positive $n$ we have $A^nu_j \in \Z^d$ for every $j=1,\ldots,N-1$. For all such $n$ we clearly have  $e^{2\pi i \langle A^nu_j,w\rangle}=1$ for every $j=1,\ldots,N-1$ so that $|\sum_{j=1}^{N-1} e^{2\pi i \langle A^nu_j,w\rangle}|=N-1>1$. On the other hand since $\lim_{k\to-\infty} A^{-k}u_j=0$ for every $j=1,\ldots,N-1$ we also have  $\lim_{n \to -\infty} |\sum_{j=1}^{N-1} e^{2\pi i \langle A^nu_j,w\rangle}|=N-1>1$. We conclude that  $|\sum_{j=1}^{N-1} e^{2\pi i \langle A^nu_j,w\rangle}|$ can equal $1$ for at most finitely many $n$, so $\mathcal{G}$ is indeed equal to the intersection of a finite collection of open, dense subsets of $\R^d$ and is therefore open and dense as needed. The result follows.\end{proof}
 The following result now completes the proof of Theorem \ref{th:main}.
  \begin{lemma}
 Let $w \in \Z^d$ be nonzero and let $(u_1,\ldots,u_N) \in \mathcal{V}_w$. Let $(T_1,\ldots,T_N)$ denote the affine iterated function system associated to $(u_1,\ldots,u_N)$ and $\nu$ the unique Borel probability measure on $\R^d$ satisfying $\nu = \frac{1}{N} \sum_{j=1}^N (T_j)_*\nu$. Then $\nu$ is singular with respect to Lebesgue measure.
 \end{lemma}
 \begin{proof}
 Fix a nonzero vector $w \in \Z^d$ and a tuple $(u_1,\ldots,u_N) \in\mathcal{V}_w$ throughout the proof, and let $(T_1,\ldots,T_N)$ be the associated iterated function system and $\nu$ the associated measure. We define the Fourier transform of $\nu$ by $\hat\nu(\xi):=\int_{\R^d} e^{i \langle x,\xi\rangle} d\nu(x)$ for every $\xi \in \R^d$. We will show that the limit $\lim_{r \to \infty} \hat\nu(2\pi (A^\transp)^r)$ exists and is nonzero, which by the Riemann-Lebesgue lemma implies that $\nu$ is not absolutely continuous with respect to Lebesgue measure. By Proposition \ref{pr:key} it will then follow that $\nu$ is singular. 

We first claim that the infinite product
\[\prod_{n=-\infty}^\infty \left(\frac{1}{N}\sum_{j=1}^N e^{2\pi i \langle A^{-n} u_j,w\rangle}\right)\]
is absolutely convergent and nonzero. Since $(u_1,\ldots,u_N) \in \mathcal{V}_w$ it is true by definition that every term in the product is nonzero. Since $A^m u_j \in \Z^d$ for all sufficiently large positive integers $m$, the terms in the product are also identically equal to $1$ for all large enough negative integers $n$. On the other hand we clearly have $\max_j |e^{2\pi i \langle A^{-n} u_j,w\rangle}-1|=O(\|A^{-n}\|)$ as $n \to +\infty$ by Lipschitz continuity. By Gelfand's formula the latter expression converges exponentially to zero as $n \to +\infty$ and this suffices to demonstrate the absolute convergence of the product. The claim is proved.

We next derive a formula for $\hat\nu(2\pi (A^\transp)^rw)$ for each $r \geq 1$. If $\xi \in \R^d$ is arbitrary then using the self-similarity relation for $\nu$,
\begin{align*}\hat\nu(\xi) &= \int_{\R^d} e^{i \langle x,\xi \rangle} d\nu(x) \\
&=\int_{\R^d}e^{i\langle x,\xi\rangle } d\left(\frac{1}{N}\sum_{j=1}^N (T_j)_*\nu\right)(x)\\
&=\frac{1}{N}\sum_{j=1}^N \int_{\R^d} e^{i\langle T_j x,\xi\rangle} d\nu(x)\\
&=\frac{1}{N}\sum_{j=1}^N \int_{\R^d}e^{i \langle A^{-1}x+u_j,\xi\rangle} d\nu(x)\\
&=\frac{1}{N}\sum_{j=1}^N e^{i \langle u_j, \xi\rangle}\int_{\R^d} e^{i\langle x,(A^\transp)^{-1}\xi\rangle} d\nu(x)=\frac{1}{N}\sum_{j=1}^N e^{i \langle u_j, \xi\rangle} \hat\nu\left((A^\transp)^{-1}\xi\right).\end{align*}
Iterating this relation $m$ times yields the formula
\[\hat\nu(\xi)= \prod_{n=0}^{m-1} \left(\frac{1}{N} \sum_{j=1}^N e^{i \langle A^{-n}u_j,\xi\rangle} \right)\hat\nu\left((A^\transp)^{-m}\xi\right)\]
valid for all $m\in \N$ and $\xi \in \R^d$. Since $\hat\nu$ is continuous and satisfies $\hat\nu(0)=1$ we deduce that for every integer $r \geq 0$
\begin{align*}\hat\nu\left(2\pi \left(A^\transp\right)^rw\right)&= \lim_{m \to \infty} \prod_{n=0}^{m-1} \left(\frac{1}{N} \sum_{j=1}^N e^{2\pi i \langle A^{r-n}u_j,w\rangle} \right)\hat\nu\left(2\pi \left(A^\transp\right)^{r-m}w\right)\\
&= \prod_{n=0}^\infty\left(\frac{1}{N} \sum_{j=1}^N  e^{2\pi i \langle A^{r-n}u_j,w\rangle}  \right)= \prod_{n=-r}^\infty\left(\frac{1}{N} \sum_{j=1}^N e^{2\pi i \langle A^{-n}u_j,w\rangle}  \right)\end{align*}
and therefore
\[\lim_{r \to \infty}\hat\nu\left(2\pi \left(A^\transp\right)^rw\right) = \prod_{n=-\infty}^\infty\left(\frac{1}{N} \sum_{j=1}^N e^{2\pi i \langle A^{-n}u_j,w\rangle}  \right)\]
which is well-defined and nonzero by our previous claim. Since $(A^\transp)^rw \to \infty$ in $\R^d$ as $r \to \infty$ it follows by the Riemann-Lebesgue lemma that $\nu$ is not absolutely continuous and the proof of the lemma is complete.
 \end{proof}
 \section{Acknowledgements}
This research was partially supported by the Leverhulme Trust (Research Project Grant RPG-2016-194). The author thanks Bal\'azs B\'ar\'any for helpful bibliographical suggestions.

%The author thanks Gerry Myerson for drawing his attention to the reference \cite{NiKo99}.

%Yang Wang is at HKUST these days and seems to do fairly different stuff.
\bibliographystyle{acm}
\bibliography{lapland}

\end{document}